\documentclass[11pt]{article}
\usepackage[a4paper,margin=1in]{geometry}
\usepackage{amsmath,amssymb,amsthm,mathtools}
\usepackage{tikz,pgfplots}
\pgfplotsset{compat=1.18}
\usetikzlibrary{arrows.meta}
\usepackage[T1]{fontenc}
\usepackage{lmodern}
\usepackage{booktabs}
\usepackage{microtype}
\usepackage[hidelinks]{hyperref}
\usepackage{enumitem}

\newtheorem{theorem}{Theorem}[section]
\newtheorem{proposition}[theorem]{Proposition}

\newtheorem{corollary}[theorem]{Corollary}
\theoremstyle{definition}
\newtheorem{definition}[theorem]{Definition}
\theoremstyle{remark}
\newtheorem{remark}[theorem]{Remark}

\newcommand{\Sone}{\mathbb S^1}

\newcommand{\dd}{\,d}
\newcommand{\Lip}{\operatorname{Lip}}

\newcommand{\st}{\operatorname{st}}
\newcommand{\Wone}{W_1}
\allowdisplaybreaks[1]
\hypersetup{pdftitle={Gauss--Chebyshev Quadrature, Circle Designs, and the Planar Minkowski Inverse Problem},pdfauthor={Congpei An}}

\title{Gauss--Chebyshev Quadrature, $\mathbb S^1$-Designs, and the Planar Minkowski Inverse Problem\\
\large A One-Dimensional Model for Design-Induced Convex Geometry}
\author{Congpei An\thanks{Email: \href{mailto:andbachcp@gmail.com}{andbachcp@gmail.com}}}
\date{}

\begin{document}
\maketitle

\begin{abstract}
We study the relation between equal-weight circle quadrature, Gauss--Chebyshev quadrature, and planar Minkowski reconstruction. The regular $N$-point grid is an $\mathbb S^1$-design of strength $N-1$; a half-step rotation of the $2M$-point grid projects to the classical $M$-point Gauss--Chebyshev rule. A common family of quadrature-error functionals gives exact identities for integration error, mixed-area error, and support-function error. For $N\ge3$, the centered Minkowski polygon $P_N$ with perimeter $2\pi$ has the exact Hausdorff distance from the unit disk $B^2$
\[
 d_H(P_N,B^2)=1-\frac{\pi}{N}\cot\frac{\pi}{N}
 =\frac{\pi^2}{3N^2}+O(N^{-4}).
\]
With Steiner centering, we obtain $d_H(K_\mu,K_\nu)\le\pi W_1(\mu,\nu)$ for balanced, nondegenerate probability measures, where $W_1$ uses geodesic distance. Its exponent is sharp even for smooth positive densities near the uniform measure. The identity $W_1(\nu_N,\sigma)=\pi/(2N)$ yields a general $O(N^{-1})$ bound for the regular grid. The sharper $O(N^{-2})$ rate follows from the sparse discrepancy spectrum and the inverse multiplier $(1-k^2)^{-1}$. Finally, a perturbation estimate gives sufficient conditions for preserving the quadratic rate and its leading constant.
\end{abstract}

\section{Introduction}
On the circle, the relation between spherical designs, equal-weight quadrature, and Minkowski reconstruction can be described explicitly. Let
\[
 u_j=(\cos\theta_j,\sin\theta_j),\qquad \theta_j=\frac{2\pi j}{N},\qquad j=0,\ldots,N-1.
\]
The regular configuration $X_N=\{u_j\}_{j=0}^{N-1}\subset \Sone$ underlies three classical constructions. First, its equal-weight quadrature rule integrates every Fourier mode of frequency less than $N$ exactly, making $X_N$ an $\Sone$-design of strength $N-1$. Second, a half-step rotation of the $2M$-point grid, followed by the substitution $x=\cos\theta$, gives the $M$-point Gauss--Chebyshev rule. Background on this angular-to-algebraic correspondence can be found in Powell~\cite{Powell1981}; related cosine grids appear in the Clenshaw--Curtis product-integration analysis of Sloan and Smith~\cite{SloanSmith1980}. Third, for $N\ge3$, assigning mass $2\pi/N$ to each direction $u_j$ prescribes the surface-area measure of a planar convex polygon. The Minkowski theorem determines this polygon up to translation, and equal side lengths together with equal angular spacing make it regular.

These constructions provide an explicit planar model for the higher-dimensional chain
\[
\begin{gathered}
 \text{spherical design}
 \longrightarrow \text{discrete surface-area measure}\\
 \longrightarrow \text{Minkowski inverse problem}
 \longrightarrow \text{convex polytope}.
\end{gathered}
\]
The basic quadrature and reconstruction statements are classical. The purpose of this note is to place them in a common error framework and derive explicit geometric estimates. The identities in~\eqref{eq:cheb-error}, Proposition~\ref{prop:mixedarea}, and~\eqref{eq:shape-error} express Gauss--Chebyshev error, mixed-area error, and support-function error through one family of circle quadrature defects.

The quantitative analysis distinguishes general inverse stability from the enhanced convergence of the regular grid. An explicit Green kernel gives a Lipschitz estimate for the inverse Minkowski map with Steiner centering (Proposition~\ref{prop:planar-stability}). A smooth family near the disk shows that the stability exponent is sharp (Proposition~\ref{prop:sharp-stability}). For the regular grid, spectral cancellation and the sparse discrepancy spectrum yield the exact $N^{-2}$ Hausdorff rate (Theorem~\ref{thm:hausdorff}). Corollary~\ref{cor:perturbation} then gives sufficient conditions on balanced perturbations for preserving this rate and its leading constant.

\section{Equal spacing and circle designs}
We identify $\Sone$ with $\{e^{i\theta}:0\le \theta<2\pi\}$ and let
\[
 \sigma=\frac{\dd\theta}{2\pi}
\]
denote normalized arc length. Throughout, $N$ is a positive integer and design strengths are nonnegative integers. The probability measure
\[
 \nu_N=\frac1N\sum_{j=0}^{N-1}\delta_{u_j}
\]
will be used for quadrature. For $N\ge3$, the corresponding surface-area measure is $\mu_N=2\pi\nu_N$.

\begin{definition}
A finite set $X=\{x_1,\dots,x_N\}\subset \Sone$ is an $\Sone$-design of strength $t$ if
\[
 \frac1N\sum_{j=1}^N p(x_j)=\int_{\Sone}p\,\dd\sigma
\]
for every restriction to $\Sone$ of a real polynomial in two variables of total degree at most $t$.
\end{definition}

This is the circle specialization of the spherical-design definition of Delsarte, Goethals, and Seidel~\cite{DGS}. On $\Sone$, these restrictions are precisely the real trigonometric polynomials of degree at most $t$. We use their complexification for Fourier calculations.

\begin{theorem}[Exact equal-weight circle rule]\label{thm:design}
For $X_N=\{e^{2\pi i j/N}:0\le j<N\}$,
\[
 \frac1N\sum_{j=0}^{N-1}q(\theta_j)
 =\frac1{2\pi}\int_0^{2\pi}q(\theta)\,\dd\theta
\]
for every trigonometric polynomial
\[
 q(\theta)=\sum_{|k|\le N-1}c_k e^{ik\theta}.
\]
Consequently, $X_N$ is an $\Sone$-design of strength $N-1$. This strength is sharp.
\end{theorem}

\begin{proof}
For an integer $k$,
\[
 \frac1N\sum_{j=0}^{N-1}e^{ik\theta_j}
 =\frac1N\sum_{j=0}^{N-1}e^{2\pi i kj/N}.
\]
The geometric sum equals $1$ when $N\mid k$ and $0$ otherwise. Hence it agrees with the circle integral for every $|k|\le N-1$. Linearity gives the quadrature formula. For sharpness, take the real function $q(\theta)=\cos(N\theta)$, the restriction of $\operatorname{Re}(x+iy)^N$. Its circle integral is zero while its discrete average is one, so the grid is not a design of strength $N$.
\end{proof}

\begin{remark}
This is the finite exactness property of the periodic trapezoidal rule; see Trefethen and Weideman~\cite{Trefethen} for its convergence theory. It is also an equal-weight, or Chebyshev-type, quadrature rule for the low-frequency trigonometric space.
\end{remark}

\section{From the circle rule to Gauss--Chebyshev quadrature}
The classical Chebyshev polynomial satisfies
\[
 T_k(\cos\theta)=\cos(k\theta).
\]
Thus cosine projection converts trigonometric exactness on the circle into algebraic exactness on $[-1,1]$. See Gautschi~\cite{Gautschi} for the Gaussian-quadrature framework.

\begin{theorem}[Projection to the Gauss--Chebyshev formula]\label{thm:cheb}
Let $M\ge1$ and define
\[
 \alpha_j=\frac{(2j-1)\pi}{2M},\qquad
 x_j=\cos\alpha_j,\qquad j=1,\ldots,M.
\]
Then, for every algebraic polynomial $p$ of degree at most $2M-1$,
\[
 \frac1\pi\int_{-1}^1\frac{p(x)}{\sqrt{1-x^2}}\,\dd x
 =\frac1M\sum_{j=1}^M p(x_j).
\]
Thus the $M$-point Gauss--Chebyshev rule is the even projection, under $x=\cos\theta$, of a half-step rotated $2M$-point equal-weight rule on $\Sone$.
\end{theorem}

\begin{proof}
Consider the $2M$ equally spaced angles
\[
 \beta_j=\frac{(2j-1)\pi}{2M},\qquad j=1,\ldots,2M.
\]
They differ from the standard $2M$-point grid only by a rotation, so Theorem~\ref{thm:design} gives exactness for all trigonometric polynomials of degree at most $2M-1$. For $p$ of algebraic degree at most $2M-1$, the function $p(\cos\theta)$ is an even trigonometric polynomial of degree at most $2M-1$. Hence
\[
 \frac1{2\pi}\int_0^{2\pi}p(\cos\theta)\,\dd\theta
 =\frac1{2M}\sum_{j=1}^{2M}p(\cos\beta_j).
\]
The left side equals
\[
 \frac1\pi\int_0^\pi p(\cos\theta)\,\dd\theta
 =\frac1\pi\int_{-1}^1\frac{p(x)}{\sqrt{1-x^2}}\,\dd x.
\]
The $2M$ cosine values occur in pairs, so the right side reduces to $M^{-1}\sum_{j=1}^M p(x_j)$.
\end{proof}

Cosine projection pairs the $2M$ circle nodes into $M$ distinct algebraic nodes. Here ``Gauss--Chebyshev'' denotes the Gaussian rule for the weight $(1-x^2)^{-1/2}$. The term ``Chebyshev-type quadrature'' refers more generally to equal-weight formulas; see Korevaar and Meyers~\cite{KorevaarMeyers}.

\subsection{One family of error functionals}\label{subsec:errorfunctional}
For a continuous $2\pi$-periodic function $f$ and a rotation angle $\alpha$, define
\[
 Q_N^\alpha f=\frac1N\sum_{j=0}^{N-1}f(\theta_j+\alpha),
 \qquad If=\int_{\Sone}f\,\dd\sigma,
 \qquad \mathcal E_N^\alpha(f)=Q_N^\alpha f-If.
\]
Write $Q_N=Q_N^0$ and $\mathcal E_N=\mathcal E_N^0$. Rotation gives
\[
 \mathcal E_N^\alpha(f)=\mathcal E_N(f(\,\cdot+\alpha)),
 \qquad
 \mathcal E_N^\alpha(e^{ik\theta})=0\quad(0<|k|<N).
\]
For the Gauss--Chebyshev nodes of Theorem~\ref{thm:cheb}, every continuous $p$ on $[-1,1]$ satisfies the exact error identity
\begin{equation}\label{eq:cheb-error}
 \frac1M\sum_{j=1}^M p(x_j)
 -\frac1\pi\int_{-1}^1\frac{p(x)}{\sqrt{1-x^2}}\,\dd x
 =\mathcal E_{2M}^{\pi/(2M)}(p\circ\cos).
\end{equation}
The half-step rotation is essential for the Gauss--Chebyshev nodes. Testing $\mathcal E_N$ against support functions gives mixed-area error (Subsection~\ref{sec:mixedarea}); testing it against translates of the inverse kernel gives support-function error (Subsection~\ref{subsec:green}). These identities relate integration and geometric errors through the same family of quadrature defects.

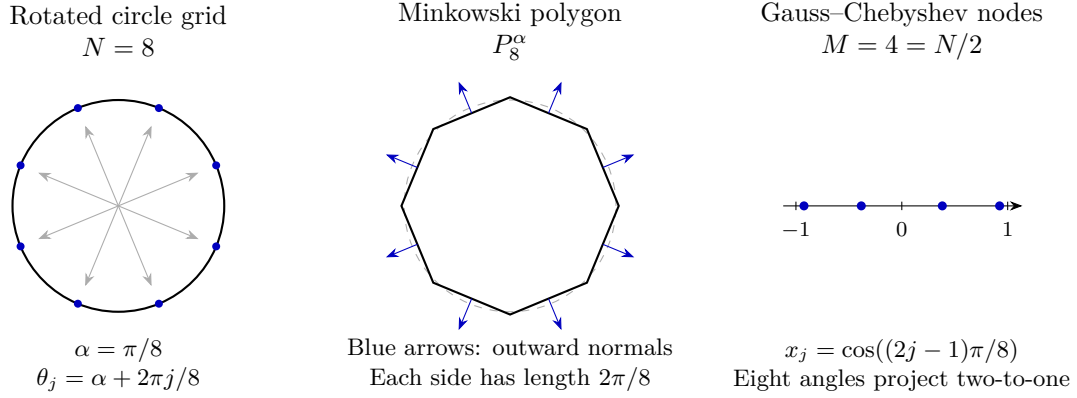
\begin{figure}[tbp]
\centering
\begin{tikzpicture}[x=1.4cm,y=1.4cm,>=Stealth,font=\small]
 \begin{scope}
  \node[align=center] at (0,1.65) {Rotated circle grid\\$N=8$};
  \draw[thick] (0,0) circle (1);
  \foreach \j in {0,...,7}{
   \pgfmathsetmacro{\ang}{22.5+45*\j}
   \draw[->,gray!65] (0,0)--(\ang:0.82);
   \fill[blue!75!black] (\ang:1) circle (1.5pt);
  }
  \node[align=center,font=\footnotesize] at (0,-1.5)
   {$\alpha=\pi/8$\\$\theta_j=\alpha+2\pi j/8$};
 \end{scope}
 \begin{scope}[xshift=5.18cm]
  \node[align=center] at (0,1.65) {Minkowski polygon\\$P_8^\alpha$};
  \draw[dashed,gray!65] (0,0) circle (1);
  \draw[thick] (45:1.026172153) \foreach \j in {2,...,8}
   {--({45*\j}:1.026172153)} --cycle;
  \foreach \j in {0,...,7}{
   \pgfmathsetmacro{\ang}{22.5+45*\j}
   \draw[->,blue!75!black] (\ang:0.948059449)--(\ang:1.27);
  }
  \node[align=center,font=\footnotesize] at (0,-1.5)
   {Blue arrows: outward normals\\Each side has length $2\pi/8$};
 \end{scope}
 \begin{scope}[xshift=10.36cm]
  \node[align=center] at (0,1.65) {Gauss--Chebyshev nodes\\$M=4=N/2$};
  \draw[->] (-1.12,0)--(1.14,0);
  \foreach \x in {-1,0,1}{
   \draw (\x,-0.04)--(\x,0.04);
   \node[below,font=\scriptsize] at (\x,-0.055) {$\x$};
  }
  \foreach \j in {1,...,4}{
   \pgfmathsetmacro{\xx}{cos((2*\j-1)*180/8)}
   \fill[blue!75!black] (\xx,0) circle (1.7pt);
  }
  \node[align=center,font=\footnotesize] at (0,-1.5)
   {$x_j=\cos((2j-1)\pi/8)$\\Eight angles project two-to-one};
 \end{scope}
\end{tikzpicture}
\caption{The half-step rotated eight-point grid, its Minkowski polygon, and the four Gauss--Chebyshev nodes obtained by cosine projection. The blue arrows on the polygon indicate its prescribed outward normal directions. The dashed circle is the unit circle.}
\label{fig:chain}
\end{figure}

\section{The planar Minkowski inverse problem}
A convex body is a compact convex set with nonempty interior. Its support function is $h_K(u)=\max_{x\in K}x\cdot u$; we also write $h_K(\theta)=h_K(\cos\theta,\sin\theta)$. For a planar convex body $K$, the surface-area measure $S_K$ records boundary length according to the outward unit normal direction. For a convex polygon with side lengths $\ell_j$ and outward normals $v_j$, one has
\[
 S_K=\sum_j \ell_j\,\delta_{v_j}.
\]
The planar Minkowski theorem says that a finite positive Borel measure $\mu$ on $\Sone$, not concentrated on an antipodal pair and satisfying
\[
 \int_{\Sone}u\,\dd\mu(u)=0,
\]
is the surface-area measure of a convex body, unique up to translation; see, for example, Schneider~\cite{Schneider}.

For the equally spaced directions, define the unnormalized measure
\[
 \mu_N=\frac{2\pi}{N}\sum_{j=0}^{N-1}\delta_{u_j}.
\]
Its total mass is $2\pi$, the perimeter of the unit disk. Reconstruction as a convex body requires $N\ge3$. For $N=1$, balance fails; for $N=2$, the measure is supported on an antipodal pair and cannot be the surface-area measure of a body with nonempty interior.

\begin{proposition}[Minkowski reconstruction of the regular polygon]\label{prop:polygon}
For $N\ge3$, the measure $\mu_N$ satisfies the Minkowski balance condition. The convex body $P_N$ with surface-area measure $S_{P_N}=\mu_N$ is, up to translation, the regular $N$-gon with side length
\[
 s_N=\frac{2\pi}{N}.
\]
We henceforth choose the representative whose center is the origin. Its circumradius and inradius are
\[
 R_N=\frac{\pi/N}{\sin(\pi/N)},
 \qquad
 r_N=\frac{\pi}{N}\cot\frac{\pi}{N}.
\]
\end{proposition}

\begin{proof}
The balance condition is
\[
 \int_{\Sone}u\,\dd\mu_N(u)
 =\frac{2\pi}{N}\sum_{j=0}^{N-1}u_j=0,
\]
which follows from degree-one design exactness. By the discrete Minkowski theorem there is a polygon whose outward normals are the $u_j$ and whose side lengths are all $2\pi/N$, unique up to translation. In cyclic order the normal directions turn by the constant angle $2\pi/N$. Hence the edges have equal lengths and equal turning angles, so the polygon is regular. The radius formulas are the elementary relations
\[
 s_N=2R_N\sin\frac\pi N,
 \qquad
 r_N=R_N\cos\frac\pi N.
\]
\end{proof}

Degree-one exactness gives balance, and the directions span the plane for $N\ge3$. These two conditions ensure existence; the grid symmetry determines the regular polygon explicitly.

\subsection{Quadrature error as mixed-area error}\label{sec:mixedarea}
For planar convex bodies $K,L$, write $V(K,L)$ for their mixed area, normalized by
\[
 |K+sL|=|K|+2sV(K,L)+s^2|L|,\qquad s\ge0.
\]
Here $|\cdot|$ denotes area and $+$ denotes Minkowski addition. The first-variation formula is
\[
 2V(K,L)=\int_{\Sone}h_L(u)\,\dd S_K(u).
\]
Since $S_{P_N}=2\pi\nu_N$ and $S_{B^2}=2\pi\sigma$, the quadrature defect from Subsection~\ref{subsec:errorfunctional} has an exact mixed-area interpretation.

\begin{proposition}[Exact mixed-area defect]\label{prop:mixedarea}
Let $N\ge3$. For every planar convex body $K$,
\[
 \boxed{\;V(P_N,K)-V(B^2,K)=\pi\,\mathcal E_N(h_K).\;}
\]
Consequently, if $h_K(\theta)$ is a trigonometric polynomial of degree at most $N-1$, then
\[
 V(P_N,K)=V(B^2,K).
\]
\end{proposition}

\begin{proof}
By the mixed-area identity and the definitions of $S_{P_N}$ and $S_{B^2}$,
\[
 V(P_N,K)=\frac12\int h_K\,\dd S_{P_N}
 =\pi Q_Nh_K,
\]
whereas
\[
 V(B^2,K)=\frac12\int h_K\,\dd S_{B^2}
 =\pi Ih_K.
\]
Subtracting gives the formula. The exactness statement follows from Theorem~\ref{thm:design}.
\end{proof}

Proposition~\ref{prop:mixedarea} identifies the quadrature error for a support function with a mixed-area error. Together with~\eqref{eq:cheb-error}, it connects the projected integration rule to a geometric observable through the same error-functional family.

\section{Exact Hausdorff convergence to the disk}
Let $B^2$ denote the unit disk and retain the origin-centered polygon $P_N$ from Proposition~\ref{prop:polygon}. Both bodies have perimeter $2\pi$. This centering fixes the translation ambiguity in Minkowski reconstruction, so the Hausdorff comparison uses uniquely normalized representatives.

Recall that for compact convex sets,
\[
 d_H(K,L)=\|h_K-h_L\|_{L^\infty(\Sone)},
\]
where $h_K$ denotes the support function. For a centered regular polygon the support function takes its minimum $r_N$ in the directions normal to the sides and its maximum $R_N$ in the directions of the vertices.

\begin{theorem}[Exact Hausdorff error]\label{thm:hausdorff}
For every $N\ge3$ and the origin-centered polygon $P_N$,
\[
 \boxed{
 d_H(P_N,B^2)
 =1-\frac{\pi}{N}\cot\frac{\pi}{N}.}
\]
Consequently,
\[
 d_H(P_N,B^2)
 =\frac{\pi^2}{3N^2}+\frac{\pi^4}{45N^4}+O(N^{-6}).
\]
Since $X_N$ has design strength $t=N-1$,
\[
 d_H(P_N,B^2)=\frac{\pi^2}{3(t+1)^2}+O(t^{-4}).
\]
\end{theorem}

\begin{proof}
Since $h_{B^2}\equiv1$,
\[
 d_H(P_N,B^2)=\max\{R_N-1,\,1-r_N\}.
\]
Put $x=\pi/N\in(0,\pi/3]$. To identify the larger term, observe that $1-r_N\ge R_N-1$ is equivalent to
\[
 2\ge x(\cot x+\csc x)=x\cot(x/2).
\]
Writing $y=x/2$, the elementary inequality $\tan y\ge y$ for $0\le y<\pi/2$ gives $y\cot y\le1$, hence $x\cot(x/2)\le2$. Therefore
\[
 d_H(P_N,B^2)=1-x\cot x.
\]
The Taylor expansion
\[
 x\cot x=1-\frac{x^2}{3}-\frac{x^4}{45}+O(x^6)
\]
gives the stated asymptotics.
\end{proof}

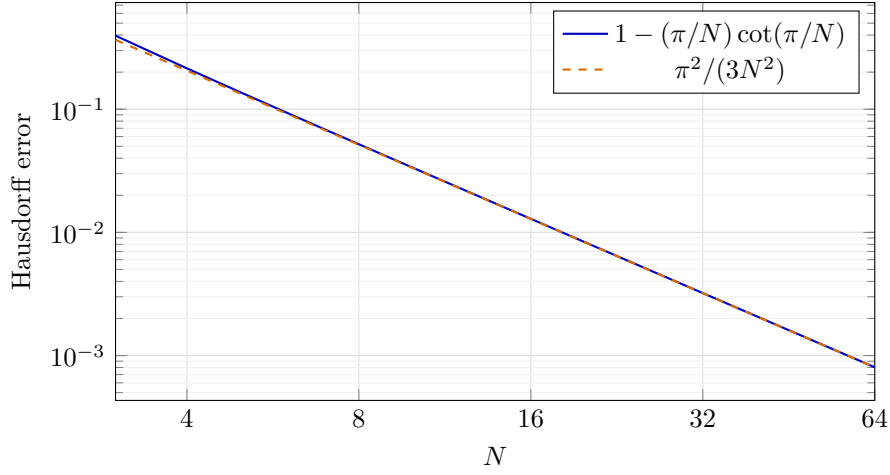
\begin{figure}[tbp]
\centering
\begin{tikzpicture}
\begin{axis}[width=0.73\textwidth,height=0.43\textwidth,
 xmode=log,ymode=log,xmin=3,xmax=64,
 xlabel={$N$},ylabel={Hausdorff error},
 xtick={4,8,16,32,64},xticklabels={4,8,16,32,64},
 grid=both,minor grid style={gray!12},major grid style={gray!25},
 legend style={font=\small,at={(0.98,0.98)},anchor=north east},
 tick label style={font=\small},label style={font=\small}]
\addplot[blue!75!black,thick] table[row sep=\\] {
N error\\
3 0.395400211921927\\
4 0.214601836602552\\
5 0.13519373402279\\
6 0.0931003178828911\\
7 0.0680593765009042\\
8 0.0519405510314801\\
9 0.0409494581263906\\
10 0.0331172009535974\\
11 0.0273379990800933\\
12 0.0229513833431467\\
13 0.0195428896520884\\
14 0.0168416606170539\\
15 0.0146645740136762\\
16 0.0128841990272246\\
17 0.011409628736115\\
18 0.0101745943917584\\
19 0.00912986160659801\\
20 0.00823823124527268\\
21 0.00747117259566432\\
22 0.00680650670917138\\
23 0.00622678143184885\\
24 0.00571811170784198\\
25 0.0052693388561269\\
26 0.00487141232611321\\
27 0.00451692908672963\\
28 0.00419978632246265\\
29 0.00391491665846244\\
30 0.00365808423461556\\
31 0.00342572615282777\\
32 0.00321482811383034\\
33 0.00302282606864746\\
34 0.00284752784526765\\
35 0.00268705024410043\\
36 0.00253976820827406\\
37 0.00240427349028582\\
38 0.00227934083990255\\
39 0.00216390018868751\\
40 0.0020570136456427\\
41 0.0019578563757956\\
42 0.00186570063026847\\
43 0.0017799023478362\\
44 0.00169988986536473\\
45 0.00162515436609689\\
46 0.00155524176661692\\
47 0.00148974580004391\\
48 0.00142830209802558\\
49 0.0013705831100248\\
50 0.00131629372718478\\
51 0.00126516750127303\\
52 0.00121696336797916\\
53 0.00117146279911462\\
54 0.00112846732072902\\
55 0.0010877963443745\\
56 0.00104928526716153\\
57 0.00101278380319281\\
58 0.000978154514722718\\
59 0.000945271516177626\\
60 0.00091401932817281\\
61 0.00088429186200456\\
62 0.000855991517914045\\
63 0.000829028382783425\\
64 0.000803319514927736\\
};
\addlegendentry{$1-(\pi/N)\cot(\pi/N)$}
\addplot[orange!85!black,dashed,thick,domain=3:64,samples=80] {pi^2/(3*x^2)};
\addlegendentry{$\pi^2/(3N^2)$}
\end{axis}
\end{tikzpicture}
\caption{The exact Hausdorff error from Theorem~\ref{thm:hausdorff} and its quadratic leading term for the centered polygon with perimeter $2\pi$. The solid curve evaluates the exact formula; the dashed curve shows its leading term. Both have asymptotic slope $-2$ on logarithmic axes.}
\label{fig:error}
\end{figure}

\subsection{Why the rate is quadratic: a Fourier multiplier identity}\label{subsec:fourierMinkowski}
For a periodic function $f$ and a finite signed measure $\eta$, we use the conventions
\[
 \widehat f(k)=\frac1{2\pi}\int_0^{2\pi}f(\theta)e^{-ik\theta}\,\dd\theta,
 \qquad \widehat\eta(k)=\int_{\Sone}e^{-ik\theta}\,\dd\eta(\theta).
\]
Thus a density $f\,d\theta$ has measure coefficient $2\pi\widehat f(k)$. Under the angular parametrization, the planar Minkowski equation is
\begin{equation}\label{eq:minkowski-distribution}
 h_K''+h_K=S_K
\end{equation}
in periodic distributions: for every smooth periodic $\psi$,
\[
 \int_0^{2\pi}h_K(\theta)(\psi''(\theta)+\psi(\theta))\,\dd\theta
 =\int_{\Sone}\psi\,\dd S_K.
\]
Translation adds $a\cos\theta+b\sin\theta$ to $h_K$. The condition $\widehat h_K(\pm1)=0$ fixes translation; it is equivalent to placing the Steiner point
\[
 \st(K)=\frac1\pi\int_0^{2\pi}h_K(\theta)(\cos\theta,\sin\theta)\,\dd\theta
\]
at the origin. For a regular polygon this is its usual center.

Since $S_{P_N}=2\pi\nu_N$ and $S_{B^2}=2\pi\sigma$, the difference $g_N=h_{P_N}-1$ solves $g_N''+g_N=2\pi(\nu_N-\sigma)$. The Fourier coefficients of the measure discrepancy are
\[
 \widehat{\nu_N-\sigma}(k)=
 \begin{cases}
  1,&k=mN,\quad m\in\mathbb Z\setminus\{0\},\\
  0,&\text{otherwise}.
 \end{cases}
\]
For $k\ne\pm1$, the Fourier equation is
$(1-k^2)\widehat g_N(k)=\widehat{\nu_N-\sigma}(k)$; the first harmonics vanish by centering. Therefore
\begin{equation}\label{eq:fourier-support}
 \boxed{\quad h_{P_N}(\theta)-1
 =-2\sum_{m=1}^{\infty}\frac{\cos(mN\theta)}{m^2N^2-1}.\quad}
\end{equation}
The coefficients are absolutely summable, so this series converges uniformly and represents the continuous support-function difference. For $N\ge3$,
\begin{equation}\label{eq:tail-bound}
 2\sum_{m\ge1}\frac1{m^2N^2-1}
 \le\frac{2}{N^2-1}\sum_{m\ge1}\frac1{m^2}
 =\frac{\pi^2}{3(N^2-1)}.
\end{equation}
This estimate bounds the entire series uniformly. At $\theta=0$, all cosines equal one and $h_{P_N}(0)=r_N$, so the bound obtained by summing the absolute coefficients is attained. Hence
\begin{equation}\label{eq:sum-identity}
 \|h_{P_N}-1\|_\infty
 =2\sum_{m\ge1}\frac1{m^2N^2-1}
 =1-\frac\pi N\cot\frac\pi N.
\end{equation}
Thus the support-function formula also gives an elementary derivation of the series sum. Three properties explain the quadratic rate: cancellation of the low frequencies, concentration of the remaining frequencies on $N\mathbb Z$, and decay of the inverse multiplier like $k^{-2}$. The uniform series estimate makes this mechanism quantitative.

\section{Quantitative stability, discrepancy, and perturbations}\label{sec:stability}
To compare reconstruction errors for general normal measures with those of the regular grid, we use geodesic Wasserstein distance and fix Steiner centering.

For $u=e^{i\theta}$ and $v=e^{i\phi}$, let
\[
 d_g(u,v)=\min_{\ell\in\mathbb Z}|\theta-\phi+2\pi\ell|.
\]
For probability measures on $\Sone$, define
\begin{equation}\label{eq:w1-definition}
 \Wone(\mu,\nu)=\inf_{\gamma\in\Pi(\mu,\nu)}\int d_g(u,v)\,\dd\gamma(u,v)
 =\sup_{\Lip_{d_g}(f)\le1}\left|\int f\,\dd(\mu-\nu)\right|.
\end{equation}
Here $\Pi(\mu,\nu)$ is the set of couplings with marginals $\mu$ and $\nu$. The second equality is Kantorovich--Rubinstein duality, with the supremum taken over real functions. Since $\mu-\nu$ has total mass zero, additive constants do not affect the integral. Thus $\Wone$ is the dual Lipschitz norm of the measure difference.

A probability measure $\mu$ is called balanced if $\int u\,d\mu(u)=0$ and nondegenerate if its support is not contained in an antipodal pair. For each such measure, let $K_\mu$ denote the unique Steiner-centered body satisfying
\[
 S_{K_\mu}=2\pi\mu,\qquad \widehat h_{K_\mu}(\pm1)=0.
\]
These bodies have perimeter $2\pi$. No quantitative separation from antipodal concentration is assumed in the planar estimate below.

\subsection{An explicit inverse kernel}\label{subsec:green}
The kernel of $D^2+1$ consists of the first harmonics. To obtain the Steiner-centered inverse, we choose a Green kernel whose first Fourier coefficients vanish. Extend the following even function periodically from $[-\pi,\pi]$:
\begin{equation}\label{eq:green}
 G(t)=-\frac{\cos t}{4\pi}
 +\left(\frac12-\frac{|t|}{2\pi}\right)\sin|t|.
\end{equation}
Direct differentiation away from $0$ gives $G''+G=-\pi^{-1}\cos t$. Moreover, $G'(0+)=1/2$ and $G'(0-)=-1/2$, while the values and derivatives match at $\pm\pi$. Consequently, in periodic distributions,
\begin{equation}\label{eq:green-distribution}
 (D^2+1)G=\delta_0-\frac1\pi\cos t.
\end{equation}
Direct integration gives $\widehat G(\pm1)=0$ in the function convention of Subsection~\ref{subsec:fourierMinkowski}; the other coefficients are
\begin{equation}\label{eq:green-fourier}
 \widehat G(k)=\frac1{2\pi(1-k^2)},\qquad k\ne\pm1.
\end{equation}
To verify the vanishing first harmonics, note that $\int_0^\pi(1/2-t/(2\pi))\sin t\cos t\,dt=1/8$, which cancels the contribution of $-\cos t/(4\pi)$.

For $0<t<\pi$,
\[
 G'(t)=-\frac{\sin t}{4\pi}
 +\left(\frac12-\frac{t}{2\pi}\right)\cos t.
\]
On $[0,\pi/2]$, $G'$ lies between $-1/(4\pi)$ and $1/2$; on $[\pi/2,\pi]$, its absolute value is at most $1/4+1/(4\pi)<1/2$. Together with evenness, periodicity, and the one-sided derivative at zero, this proves
\begin{equation}\label{eq:green-lip}
 \Lip_{d_g}(G)=\frac12.
\end{equation}

\begin{proposition}[Planar Lipschitz stability]\label{prop:planar-stability}
For any two balanced, nondegenerate probability measures $\mu,\nu$ as above,
\begin{equation}\label{eq:stability}
 \boxed{\quad d_H(K_\mu,K_\nu)\le\pi\Wone(\mu,\nu).\quad}
\end{equation}
With the stated perimeter and Steiner normalizations, the constant is independent of the number and sizes of atoms and of the concentration of the measures.
\end{proposition}

\begin{proof}
Let $\eta=\mu-\nu$ and $g=h_{K_\mu}-h_{K_\nu}$. The total mass and the first moments of $\eta$ vanish. Equation~\eqref{eq:green-distribution} shows that
\[
 \widetilde g(\theta)=2\pi\int_{\Sone}G(\theta-\phi)\,\dd\eta(\phi)
\]
satisfies $\widetilde g''+\widetilde g=2\pi\eta$: the cosine correction vanishes by balance. Its first harmonics vanish by~\eqref{eq:green-fourier}. The same equation and normalization hold for $g$. Their difference solves $v''+v=0$ with zero first harmonics, hence $g=\widetilde g$. Using~\eqref{eq:w1-definition} and~\eqref{eq:green-lip},
\[
 |g(\theta)|\le2\pi\Lip_{d_g}(G)\Wone(\mu,\nu)=\pi\Wone(\mu,\nu).
\]
Taking the supremum proves~\eqref{eq:stability}.
\end{proof}

Applying the kernel representation to the regular grid gives, for $N\ge3$, the exact support-function error identity
\begin{equation}\label{eq:shape-error}
 h_{P_N}(\theta)-1
 =2\pi\mathcal E_N\bigl(G(\theta-\,\cdot)\bigr).
\end{equation}
Equations~\eqref{eq:cheb-error} and~\eqref{eq:shape-error}, together with Proposition~\ref{prop:mixedarea}, give the integration, mixed-area, and support-function interpretations of the same quadrature-error family.

A termwise Fourier estimate does not yield the Lipschitz bound directly. For $k\ne0,\pm1$, duality and the Fourier equation give
\[
 |\widehat g(k)|\le\frac{|k|}{|1-k^2|}\Wone(\mu,\nu).
\]
For a complex Fourier mode, the bound follows by taking its real part with a phase that realizes the modulus. These upper bounds are not summable over $k$, so absolute Fourier summation does not give~\eqref{eq:stability}. Estimating the Lipschitz norm of the inverse kernel avoids this loss.

For related reconstruction and stability results in higher dimensions, see Abdallah and M\'erigot~\cite{AbdallahMerigot} and B\"or\"oczky, Machado, and Ramos~\cite{BMR2026}. The latter estimate Hausdorff distance, after optimizing translation, by a power of a dual-convex discrepancy under quantitative nondegeneracy assumptions. That discrepancy is bounded by $\Wone$, since support functions of bodies contained in the unit ball are $1$-Lipschitz. The explicit planar kernel gives a direct estimate for fixed Steiner-centered representatives.

\subsection{Exact discrepancy of the regular grid}
\begin{proposition}[Exact $W_1$ discrepancy]\label{prop:w1}
For every $N\ge1$, with the geodesic metric,
\begin{equation}\label{eq:grid-w1}
 \Wone(\nu_N,\sigma)=\frac\pi{2N}.
\end{equation}
\end{proposition}

\begin{proof}
Transport each arc Voronoi cell, of length $2\pi/N$ and $\sigma$-mass $1/N$, to its central node. Its cost summed over the grid is
\[
 \frac N{2\pi}\int_{-\pi/N}^{\pi/N}|\theta|\,\dd\theta=\frac\pi{2N}.
\]
This is a coupling with marginals $\sigma$ and $\nu_N$, proving the upper bound. Conversely, $f(u)=d_g(u,X_N)$ is $1$-Lipschitz, vanishes on $X_N$, and satisfies $\int f\,d\sigma=\pi/(2N)$. The dual formula~\eqref{eq:w1-definition} gives the matching lower bound.
\end{proof}

For $N\ge3$, Propositions~\ref{prop:planar-stability} and~\ref{prop:w1} give
\[
 d_H(P_N,B^2)\le\frac{\pi^2}{2N}.
\]
For the regular grid, Theorem~\ref{thm:hausdorff} gives the sharper $N^{-2}$ rate. The additional order of convergence follows from the sparse grid spectrum and the uniform series bound~\eqref{eq:tail-bound}.

\subsection{Sharpness and balanced perturbations}
\begin{proposition}[Sharpness of the stability exponent]\label{prop:sharp-stability}
For $0<|a|<1$, let $d\mu_a=(1+a\cos2\theta)\,d\sigma$. Then
\begin{equation}\label{eq:sharp-example}
 d_H(K_{\mu_a},B^2)=\frac{|a|}{3},\qquad
 \Wone(\mu_a,\sigma)=\frac{|a|}{\pi}.
\end{equation}
In particular, no estimate $d_H(K_\mu,B^2)\le C\Wone(\mu,\sigma)^\beta$ with $\beta>1$ and fixed $C$ holds uniformly near $\sigma$, even among smooth positive densities converging uniformly to $1$.
\end{proposition}

\begin{proof}
The measure is positive and balanced. Its centered support function is
\[
 h_a(\theta)=1-\frac a3\cos2\theta,
 \qquad h_a''+h_a=1+a\cos2\theta>0.
\]
The positivity of $h_a''+h_a$ gives a smooth strictly convex body, and Minkowski uniqueness identifies it with $K_{\mu_a}$. This proves the Hausdorff formula. Every periodic $1$-Lipschitz function is absolutely continuous with $|f'|\le1$ almost everywhere. Integration by parts gives
\[
 \left|\int f\,d(\mu_a-\sigma)\right|
 =\left|-\frac a{4\pi}\int_0^{2\pi}f'(\theta)\sin2\theta\,d\theta\right|
 \le\frac{|a|}{\pi}.
\]
Equality is attained by a periodic function with derivative
$f'(\theta)=-\operatorname{sgn}(a\sin2\theta)$ almost everywhere: this derivative has zero mean, so its primitive is periodic and $1$-Lipschitz for $d_g$. Duality proves the second identity. Thus $d_H=(\pi/3)\Wone$ along this family as $a\to0$, which excludes every exponent $\beta>1$.
\end{proof}

The sharpness statement concerns the exponent. The optimal Lipschitz constant in Proposition~\ref{prop:planar-stability} is not determined here.

\begin{corollary}[Stability of the quadratic grid rate]\label{cor:perturbation}
Let $N\ge3$ and let $\rho_N$ be a balanced probability measure not supported on an antipodal pair. If $\varepsilon_N=\Wone(\rho_N,\nu_N)$, then
\begin{equation}\label{eq:perturbation}
 \left|d_H(K_{\rho_N},B^2)-\left(1-\frac\pi N\cot\frac\pi N\right)\right|
 \le\pi\varepsilon_N.
\end{equation}
Consequently, $\varepsilon_N=O(N^{-2})$ implies $d_H(K_{\rho_N},B^2)=O(N^{-2})$; if $\varepsilon_N=o(N^{-2})$, then
\[
 N^2d_H(K_{\rho_N},B^2)\longrightarrow\frac{\pi^2}{3}.
\]
If $\rho_N=N^{-1}\sum_{j=0}^{N-1}\delta_{\widetilde u_j}$ satisfies the stated assumptions, matching each perturbed node to its original node gives
$\varepsilon_N\le N^{-1}\sum_{j=0}^{N-1}d_g(\widetilde u_j,u_j)$.
\end{corollary}

\begin{proof}
The reverse triangle inequality for Hausdorff distance and~\eqref{eq:stability} give
\[
 |d_H(K_{\rho_N},B^2)-d_H(P_N,B^2)|
 \le d_H(K_{\rho_N},P_N)\le\pi\varepsilon_N.
\]
Use Theorem~\ref{thm:hausdorff}. The node estimate follows from the coupling that matches each perturbed node to its original node.
\end{proof}

The balance condition must be checked after perturbing the nodes: independent angular changes may violate it and thus prevent Minkowski reconstruction. The corollary gives sufficient conditions for preserving the quadratic rate; necessity of these conditions is not asserted.

\medskip
\noindent\textbf{Final Remark.}
The circle model separates the role of exactness from that of inverse stability. The common error-functional representation links quadrature to mixed area and support functions. The general inverse estimate is linear in $W_1$, with sharp exponent, while the sparse spectrum of the regular grid yields quadratic Hausdorff convergence. Corollary~\ref{cor:perturbation} quantifies the persistence of this rate under balanced perturbations.

In higher dimensions, the existence statement requires nondegeneracy as well as balance. Let $X=\{x_j\}_{j=1}^N\subset\mathbb S^d$, $d\ge2$, be a spherical $t$-design with $t\ge2$, and let $\nu_X=N^{-1}\sum_j\delta_{x_j}$. Exactness through degree two gives
\[
 \int x\,d\nu_X(x)=0,\qquad
 \int xx^{\mathsf T}\,d\nu_X(x)=\frac1{d+1}I_{d+1}.
\]
For any nonzero $v$, $\int(v\cdot x)^2\,d\nu_X=|v|^2/(d+1)>0$, so the support is not contained in a great subsphere. Minkowski's theorem therefore produces a full-dimensional polytope $P_X$ with
$S_{P_X}=|\mathbb S^d|\nu_X$, unique up to translation. Degree-one exactness alone is insufficient: the four equatorial points $\{(\pm1,0,0),(0,\pm1,0)\}\subset\mathbb S^2$ form a $1$-design but fail the spanning condition.

The higher-dimensional setting draws on the existence theory of spherical designs~\cite{BRV} and quantitative Minkowski stability~\cite{BMR2026}. Design-induced polytopes, isotropy transfer, and quantitative sphericity are studied in~\cite{An2026}. The present note provides an explicitly solvable planar model and identifies its connection to Gauss--Chebyshev quadrature. In ambient dimension at least three, the smooth reconstruction equation is nonlinear,
\[
 \det\bigl(\nabla_{\mathbb S^d}^2h+hI\bigr)=f,
 \qquad dS_K=f\,d\omega,
\]
where $d\omega$ is unnormalized spherical area and $I$ acts on the tangent space. The polygon formula and linear inverse kernel are specific to the plane; they imply no higher-dimensional $N^{-2}$ reconstruction rate.

\end{document}